\documentclass[11pt, a4paper]{article}%
\usepackage{amsmath}%
\usepackage{amsfonts}%
\usepackage{amssymb}%
\usepackage{graphicx}
\usepackage{color}

\newtheorem{theorem}{Theorem}

\newtheorem{corollary}[theorem]{Corollary}

\newtheorem{definition}[theorem]{Definition}
\newtheorem{example}[theorem]{Example}

\newtheorem{lemma}[theorem]{Lemma}
\newtheorem{notation}[theorem]{Notation}

\newtheorem{proposition}[theorem]{Proposition}
\newtheorem{remark}[theorem]{Remark}

\newenvironment{proof}[1][Proof]{\noindent\textbf{#1.} }{\
\rule{0.5em}{0.5em}}

\def\C{\mathbb{C}}
\def\N{\mathbb{N}}

\def\Q{\mathbb{Q}}

\def\cF{\mathcal{F}}

\def\d{\displaystyle}

\def\vec{\mathbf}

\def\square{\ \rule{0.5em}{0.5em}}

\begin{document}

\title{Effective Differential Nullstellensatz for Ordinary DAE Systems with Constant Coefficients \thanks{Partially supported by the following grants: UBACYT 20020110100063 (2012-2015) / Math-AmSud SIMCA ``Implicit Systems, Modelling and Automatic Control" (2013-2014) / Subsidio para Investigador CONICET 4541/12 (2013) (P.S).}}

\author{Lisi D'Alfonso$^\natural$ \and Gabriela Jeronimo$^\sharp$ \and Pablo Solern\'o$^\sharp$\\[4mm]
{\small $\natural$ Departamento de Ciencias Exactas, Ciclo B\'asico Com\'un, Universidad de Buenos Aires,} \\{\small Ciudad Universitaria, 1428, Buenos Aires, Argentina}\\[2mm]
{\small $\sharp$ Departamento de Matem\'atica and IMAS, UBA-CONICET,} \\{\small Facultad de Ciencias Exactas y Naturales,Universidad de Buenos Aires,}\\ {\small Ciudad Universitaria, 1428, Buenos Aires, Argentina}\\[3mm]
{\small E-mail addresses: lisi@cbc.uba.ar, jeronimo@dm.uba.ar, psolerno@dm.uba.ar}
}
\maketitle

\begin{abstract}
We give upper bounds for the differential Nullstellensatz in the case of ordinary systems of differential algebraic equations over any field of constants $K$ of characteristic $0$.
Let $\vec{x}$ be a set of $n$ differential variables, $\vec{f}$ a finite family of differential polynomials in the ring $K\{\vec{x}\}$ and  $f\in K\{\vec{x}\}$ another polynomial which vanishes at every solution of the differential equation system $\vec{f}=0$ in any differentially closed field containing $K$.
Let $d:=\max\{\deg(\vec{f}), \deg(f)\}$ and $\epsilon:=\max\{2,{\rm{ord}}(\vec{f}), {\rm{ord}}(f)\}$.
We show that $f^M$ belongs to the algebraic ideal generated by the successive derivatives of $\vec{f}$ of order at most $L = (n\epsilon d)^{2^{c(n\epsilon)^3}}$, for a suitable universal constant $c>0$,  and $M=d^{n(\epsilon +L+1)}$. The previously known bounds for $L$ and $M$ are not elementary recursive.

\end{abstract}

\section{Introduction}

In 1890, D. Hilbert states  his famous result, currently known as
Hilbert's Nullstellensatz: if $ k $ is a field and $f_1, \ldots f_s,
h$ are multivariate polynomials such that every zero of the $f_i$'s,
in an algebraic closure of $k$, is a zero of $h$, then some power of
$h$ is a linear combination of the $f_i$'s with polynomial
coefficients. In particular if $f_1, \ldots , f_s$ have no common
zeros, then there exist polynomials $h_1, \ldots h_s$ such that
$1=h_1f_1+\ldots +h_sf_s$.  The classical proofs give
no information about the polynomials $h_i$, for instance, they give
no bound for their degrees. The knowledge of such bounds yields a simple
way of determining whether the algebraic variety $\{f_1=0, \ldots,
f_s=0\}$ is empty.  G. Hermann, in 1925, first addresses this
question in \cite{GH}  where she obtains a bound for the degrees of
the $h_i$'s double exponential in the number of variables. In the
last 25 years, several authors have shown bounds single exponential
in the number of variables  (for a survey of the first results  see \cite{BS91} and for more recent improvements see
\cite{Jelonek05}).

In 1932, J. F. Ritt  in \cite{Ritt32} introduces for the first time
the differential  version of Hilbert's Nullstellensatz in the
ordinary context: {\it{Let $f_1, \ldots f_k, h$ be multivariate
differential polynomials  with coefficients in an ordinary
differential field $\cF$. If every zero of the system $f_1, \ldots,
f_k$ in any extension of $\cF$ is a zero of $h$, then some power of
$h$ is a linear combination of the $f_i$'s and a certain number of
their derivatives, with polynomials as coefficients. In particular
if the $f_i$'s do not have  common zeros, then a combination of
$f_1, \ldots f_k$ and their derivatives of various orders equals
unity.}}

In fact, Ritt considers only  the case of differential polynomials
with coefficients in a differential field $\cF$ of meromorphic
functions in an open set of the complex plane. Later, H.W.
Raudenbusch, in \cite{Raud34}, proves this result for polynomials
with coefficients in any abstract ordinary differential field of
characteristic $0$. In 1952, A. Seidenberg gives a proof for
arbitrary characteristic (see \cite{Seid52}) and,  in 1973, E.
Kolchin, in his book \cite{Kolchin},  proves the generalization of
this result to differential polynomials with coefficients in an
arbitrary, not necessarily ordinary, differential field.

None of the proofs mentioned above gives a constructive method for
obtaining admissible values of the power of the polynomial $h$ that
is a combination of the $f_i$'s and their derivatives, or for the
number of these derivatives.  A  bound for these orders of derivation allows us to work
in a polynomial ring in finitely many variables and invoke the results of the algebraic
Nullstellensatz in order to determine whether or not a differential
system has a solution.  A first step in this direction was given by
R. Cohn in \cite{Cohn}, where he proves the existence of the power of
the polynomial $h$ through a process that it is known to have only a
finite number of steps. In \cite{Seid56},  Seidenberg studies  this
problem in the case of ordinary and partial differential systems, proving the existence of functions in terms of the parameters of the input polynomials which describe the order of the derivatives involved; however, no bounds are explicitly shown there.
 The first known bounds on this subject are given by O.
Golubitzky et al.~in \cite{GKOS}, by means of rewriting techniques.
Their general upper bounds are stated in terms of the Ackermann function and, in particular, they are not primitive recursive (see \cite[Theorem 1]{GKOS}). If the number of derivations is fixed (for instance, in the case of ordinary differential equations), the bounds become primitive recursive; however, they are  not elementary recursive, growing faster than any tower of exponentials of fixed height.

\bigskip

The present paper deals with effective aspects of the {\it ordinary}
differential Nullstellensatz over the field $\C$ of complex numbers or, more generally, over arbitrary fields of constants of characteristic $0$. Our main result, which can be found in
Corollary \ref{nssstrong} (see also Section \ref{general} for the general case) states a doubly exponential upper bound for the number of required differentiations:

\vspace{5mm}

\noindent\textbf{Theorem} \ {\it  Let $\vec{x}:=x_1, \ldots, x_n$
and $\vec{f}:=f_1, \ldots, f_s$ be differential  polynomials in
$\C\{\vec{x}\}$. Suppose that $f\in\C\{\vec{x}\}$ is a differential
polynomial such that every solution of the differential system
$\vec{f}=0$ in any differentially closed field containing $\C$ satisfies also ${f}=0$. Let
$d:=\max\{\deg(\vec{f}), \deg(f)\}$ and
$\epsilon:=\max\{2,{\rm{ord}}(\vec{f}), {\rm{ord}}(f)\}.$
Then $f^{M}\in (\vec{f}, \ldots, \vec{f}^{(L)})$ where $L\le(n\epsilon d)^{2^{c(n\epsilon)^3}}$, for a universal constant $c>0$,  and
$M=d^{n(\epsilon+L+1)}$. 
}

\vspace{5mm}

In particular, this theorem, combined with known degree bounds for the polynomial coefficients in a representation of $1$ as a linear combination of given generators of trivial algebraic ideals, allows the construction of an algorithm
to decide whether a differential system has a solution with a triple
exponential running time. From a different approach, an algorithm
for this decision problem, with similar complexity, can be deduced
as a particular case of the quantifier elimination method for
ordinary differential equations proposed by D. Grigoriev in
\cite{Grigoriev87}.

Our approach focuses mainly on the  consistency problem for \emph{first order
semiexplicit ordinary systems over $\C$}, namely differential systems of the
type of the system (\ref{elsistema}) below. An iterative process of
prolongation and projection, together with several tools from
effective commutative algebra and algebraic geometry, is applied in
such a way that in each step the dimension of the algebraic
constraints decreases until a reduced, zero-dimensional situation is
reached. The bounds for this particular case are computed directly.
Then, by means of a recursive reconstruction, we are able to obtain
a representation of $1$ as an element of the differential ideal
associated to the original system.
 The results for any arbitrary ordinary DAE system over $\C$ are deduced through the classical method of  reduction of order, the algebraic Nullstellensatz, and the Rabinowitsch trick. The generalization from $\C$ to an arbitrary field of constants of characteristic $0$ is achieved by means of standard arguments of field theory.

This paper is organized as follows. In Section
\ref{preliminares} we introduce some basic tools and previous results from effective commutative algebra and algebraic geometry and some basic notions and notations  from differential algebra.
In Section \ref{dim0} we address  the case of semiexplicit systems with   reduced $0$-dimensional algebraic constraints. In Section \ref{ODEarbitrary} we show the process that reduces  the arbitrary dimensional case to the reduced $0$-dimensional one and then  recover the information for the original system. In Section \ref{generalcase}, the general case of an arbitrary ordinary DAE system over $\C$ is considered. Finally, in Section \ref{general} we extend the previous results to any arbitrary base field $K$ of characteristic $0$ considered as a field of constants.

\section{Preliminaries} \label{preliminares}

In this section we recall some definitions and results from effective commutative algebra and algebraic geometry and introduce the notation and basic notions from differential algebra used throughout the paper.

\subsection{Some tools from Effective Commutative Algebra and Algebraic Geometry} \label{basic}

Throughout the paper we will need several results from effective commutative algebra and algebraic geometry. We recall them here in the precise formulations we will use.

Before proceeding, we introduce some notation.  Let $\vec{x}=x_1, \ldots , x_n$ be a set of variables and  $\vec{f}=f_1, \ldots, f_s$ polynomials in $\C[\vec{x}]$.
We will write $V(\vec{f})$ for the algebraic variety in $\C^n$ defined by $\{\vec{f}=0\}=\{x\in \C^{n}: f_1(x)=0, \ldots, f_s(x)=0\}$. If $V\subset \C^n$ is an algebraic variety, $I(V)$ will denote the vanishing ideal of the variety $V$, that is $I(V)=\{f\in \C[\vec{x}]: f(x)=0 \ \forall \, x\in V\}$.

One of the results we will apply is an effective version of the strong Hilbert's Nullstellensatz (see for instance \cite[Theorem 1.3]{Jelonek05}):

\begin{proposition}\label{prop:exprad}
Let $f_1,\dots, f_s\in {\C}[x_1,\dots, x_n]$ be polynomials of degrees bounded by $d$, and let $I=(f_1,\dots, f_s)\subset {\C}[x_1,\dots, x_n]$. Then $(\sqrt{I})\,^{d^n}\subset I$.
\end{proposition}

In addition, we will need estimates for the degrees of generators of the (radical) ideal of an affine variety $V\subset \C^n$.

A classical result due to Kronecker \cite{Kronecker1882} states that any algebraic variety in $\C^n$ can be defined by $n+1$ polynomials in $\C[x_1,\dots, x_n]$. Moreover, these $n+1$ polynomials can be chosen to be $\Q$-linear combinations of any finite set of polynomials defining the variety. In \cite[Proposition 3]{Heintz83}, a refined version of Kronecker's theorem is proved for irreducible affine varieties. In this version the degree of the $n+1$ defining polynomials is bounded by the \emph{degree} of the variety. We recall that the degree of an irreducible algebraic variety is defined as the number of points in the intersection of the variety with a generic linear variety of complementary dimension; for an arbitrary algebraic variety, it is defined as the sum of the degrees of its irreducible components (see \cite{Heintz83}).

Kronecker's theorem with degree bounds can be extended straightforwardly to an arbitrary (not necessarily irreducible) algebraic variety $V\subset \C^n$ by considering equations of degree $\deg(C)$ for each irreducible component $C$ of $V$ and multiplying them in order to obtain a finite family of polynomials of degree $\sum_{C} \deg(C)=\deg(V)$ defining $V$:

\begin{proposition}\label{n+1equations}
Let $V\subset \C^n$ be an algebraic variety. Then, there exist $n+1$ polynomials of degrees at most $\deg(V)$ whose set of common zeros in $\C^n$ is $V$.
\end{proposition}

In order to obtain upper bounds for the number and degrees of generators of the ideal of $V$,  we will apply the following estimates, which follow from the algorithm for the computation of the radical of an ideal presented in \cite[Section 4]{Laplagne06} (see also \cite{KL91b}, \cite{KL91a}) and estimates for the number and degrees of polynomials involved in Gr\"obner basis computations (see, for instance, \cite{Dube90}, \cite{MM84} and \cite{Giusti84}):

\begin{proposition} \label{radicalgen}
 Let $I=(f_1,\dots, f_s)\subset {\C}[x_1,\dots, x_n]$ be an ideal  generated by $s$ polynomials of degrees at most $d$ that define an algebraic variety of dimension $r$ and let $\nu=\max\{1, r\}$. Then, the radical ideal $\sqrt{I}$ can be generated by $(sd)^{2^{O(\nu n)}}$ polynomials of degrees at most $(sd)^{2^{O(\nu n)}}$.
 \end{proposition}

Combining Propositions \ref{n+1equations} and \ref{radicalgen}, we have:

\begin{proposition}\label{prop:radbound}
Let $V\subset \C^n$ be an algebraic variety of dimension $r$ and degree $D$ and $\nu=\max\{1, r\}$. Then, the vanishing ideal of $V$ can be generated by $(nD)^{2^{O(\nu n)}}$ polynomials of degrees at most  $(n D)^{2^{O(\nu n)}}$.
\end{proposition}

Finally, in order to use the bounds in the previous proposition, we will need to compute upper bounds for the degrees of algebraic varieties. To this end, we will apply the following  B\'ezout type bound, taken from \cite[Proposition 2.3]{HS82}:

\begin{proposition}\label{prop:bezbound} Let $V\subset \C^n$ be an algebraic variety of dimension $r$ and degree $D$, and let $f_1,\dots, f_s\in \C[x_1,\dots, x_n]$ be polynomials of degree at most $d$. Then,
$$\deg(V\cap V(f_1,\dots, f_s))\le D d^r.$$

\end{proposition}

\subsection{Basic notions  from Differential Algebra}

If $\vec{z}:=z_1,\ldots ,z_\alpha$ is a set of $\alpha$ indeterminates,
the ring of differential polynomials is denoted by $\C\{z_1,\ldots ,z_\alpha\}$ or simply
$\C\{\vec{z}\}$ and is defined as the commutative
polynomial ring $ \C[z_j^{(p)}, 1\le j \le \alpha, \ p\in \N_0]$
(in infinitely many indeterminates), with the derivation $\delta(z_j^{(i)}) = z_j^{(i+1)}$, that is,
$z_j^{(i)}$ stands for the $i$th derivative of $z_j$ (as
usual, the first derivatives are also denoted by $\dot z_j$).
We write $\vec{z}^{(p)}:=\{z^{(p)}_1,\ldots,z^{(p)}_\alpha\}$ and
$\vec{z}^{[p]}:=\{\vec{z}^{(i)},\ 0\le i\le p\}$ for every $p\in \N_0$.

For a differential polynomial $h$ lying in the differential
polynomial ring $\C\{ \vec{z}\}$ the \emph{successive total derivatives of $h$} are:
$$\begin{array}{rcl}
h^{(0)}&:=& h \\
h^{(p)}&:=& \displaystyle{\sum_{i\in
\mathbb{N}_0,1\le j\le \alpha} \dfrac{\partial h^{(p-1)}}{\partial
z_j^{(i)}}\ z_j^{(i+1)}}, \quad \hbox{ for } p\ge 1. \qquad
\end{array}$$
The \emph{order of $h\in \C\{\vec{z}\}$ with respect to $z_j$} is  ${\rm ord}(h,z_j) :=
\max\{i \in \N_0 : z_j^{(i)} \hbox{ appears in } h\}$, and the
\emph{order of $h$} is ${\rm ord}(h) := \max\{{\rm ord}(h,z_j) :
1\le j \le \alpha\}$.

Given a finite set of differential polynomials $\vec{h}= h_1,\dots,
h_\beta \in \C\{ \vec{z}\}$, we write $[\vec{h}]$ to denote the smallest
\emph{differential} ideal of $\C\{\vec{z}\}$ containing $\vec{h}$ (i.e.~the
smallest ideal containing the polynomials $\vec{h}$ and all their derivatives
of arbitrary order).
For every $i\in
\N$, we write
$\vec{h}^{(i)}:=h_1^{(i)},\ldots,h_\beta^{(i)}$ and $\vec{h}^{[i]}:=\vec{h},\vec{h}^{(1)},\ldots ,\vec{h}^{(i)}$.

\section{The case of ODE's with $0$-dimensional reduced algebraic constraint}\label{dim0}

\subsection{An introductory case: univariate ODE's}

We start by considering the simple case of trivial \emph{univariate} differential ideals contained in $\C\{x\}$
where $x$ is a single differential variable.

Suppose that the trivial ideal is presented by two generators $\dot{x}-f(x)$ and $g(x)$, without common differential solutions, where $f$ and $g$ are polynomials in $\C[x]$.

Since we assume that there exists a representation of $1$ as a combination of $\dot{x}-f(x)$ and $g(x)$ and suitable derivatives of them, by replacing in such a representation all derivatives $x^{(i)}$ by $0$ for $i\ge 1$, we deduce that the univariate polynomials $f$ and $g$ are relatively prime in $\C[x]$.

Let $1=pf+qg$ be an identity in $\C[x]$; therefore, we have
\begin{equation} \label{single x} 1=-p(x)(\dot{x}-f(x))+q(x)g(x)+p(x)\dot{x}.
\end{equation}

On the other hand, if we \emph{assume that $g$ is square-free}, from a relation $1=a(x)g(x)+b(x)\dfrac{\partial g}{\partial x}(x)$ we deduce $\dot{x}=\dot{x}a(x)g(x)+b(x)\dot{x}\dfrac{\partial g}{\partial x}(x)=\dot{x}a(x)g(x)+b(x)\dot{g}$.
Thus, replacing $\dot{x}$ in (\ref{single x}) we have:
\[ 1=-p(x)(\dot{x}-f(x))+(q(x)+p(x)\dot{x}a(x))g(x)+p(x)b(x)\dot{g}.
\]
In other words, we need at most one derivative of $g$ in order to write $1$ as combination of the derivatives of the generators $\dot{x}-f(x)$ and $g(x)$.

\bigskip

 Let us remark that a similar argument can be applied also if $g$ is not assumed square-free with the aid of the Fa\`a di Bruno formula (see for instance \cite{Johnson02}) which describes each differential polynomial $g^{(i)}$ as a $\Q$-linear combination of products of $x^{(j)}$, $j\le i$, and successive derivatives $\dfrac{{\partial}^kg}{\partial x^k}$ up to order $i$. In this case it is not difficult to show that the maximum number of derivatives of the input equations which allow us to write $1$ can be bounded \textit{a priori} by the smallest $k$ such the first $k$ derivatives of $g$ are relatively prime and, moreover, no derivatives of $\dot{x}-f$ of positive order are needed. Since we do not make use of this result,  we have not included a complete proof here.

\subsection{The multivariate case}\label{multivardim0}

Now we consider the case of an arbitrary number of variables. Suppose that $\vec{x}=x_1,\ldots,x_n$ and $\vec{u}=u_1,\ldots,u_m$ are independent differential variables. Let $\vec{f}=f_1,\ldots,f_n\in \C[\vec{x},\vec{u}]$ and $\vec{g}=g_1,\ldots,g_s\in \C[\vec{x},\vec{u}]$ be polynomials such that the (polynomial) ideal $(\vec{g})\subseteq \C[\vec{x},\vec{u}]$ is \emph{radical and $0$-dimensional}.

Suppose that the differential ideal generated by the $n+s$ polynomials $\dot{\vec{x}}-\vec{f}$ and $\vec{g}$ is the whole differential ring $\C\{\vec{x},\vec{u}\}$ (i.e. $1\in [\dot{\vec{x}}-\vec{f},\vec{g}]$).
The goal of this subsection is to show that the order of derivatives of the generators which allows us to write $1$ as a combination of them is at most $1$ (see Proposition \ref{dimension 0} below).

Under these assumptions, as in the previous section, we deduce that the polynomial ideal $(\vec{f},\vec{g})$ is the ring $\C[\vec{x},\vec{u}]$; hence, we have an algebraic identity $1=\vec{p}\cdot\vec{f}+\vec{q}\cdot\vec{g}$ for suitable $n+s$ polynomials $\vec{p},\vec{q}\in \C[\vec{x},\vec{u}]$. Thus,
\begin{equation} \label{single x u}
1=-\vec{p}\cdot(\dot{\vec{x}}-\vec{f})+\vec{q}\cdot\vec{g}+\vec{p}\cdot \dot{\vec{x}}.
\end{equation}

Since the polynomial ideal $(\vec{g})$ is assumed to be radical and $0$-dimensional, for each variable $x_j$ there exists a nonzero square-free univariate polynomial $h_j\in \C[x_j]$
such that $h_j(x_j)\in (\vec{g})\subseteq \C[\vec{x},\vec{u}]$. The square-freeness of $h_j$ implies that the relation  $1=a_jh_j+b_j\dfrac{\partial h_j}{\partial x_j}$ holds in the ring $\C[x_j]$ for suitable polynomials $a_j,b_j\in \C[x_j]$ and then, after multiplying by $\dot{x_j}$ we obtain the identities \begin{equation} \label{x dot}
\dot{x_j}=a_j\dot{x_j}h_j+b_j\dot{h}_j,\qquad \textrm{for}\ j=1,\ldots,n.\end{equation}
On the other hand, each polynomial ${h}_j$ can be written as a linear combination of the polynomials $\vec{g}$  with coefficients in $\C[\vec{x},\vec{u}]$, which induces by derivation a representation of its derivative $\dot{{h}_j}$ as a linear combination of $\vec{g},\dot{\vec{g}}$ with coefficients in $\C[\vec{x},\vec{u},\dot{\vec{x}},\dot{\vec{u}}]$.
Replacing ${h}_j$ and $\dot{{h}}_j$ in (\ref{x dot}) by these combinations and then replacing $\dot{\vec{x}}$ in (\ref{single x u}), we conclude:
\begin{proposition} \label{dimension 0}
With the previous notations and assumptions we have
\[
 1\in [\dot{\vec{x}}-\vec{f},\vec{g}]\subseteq \C\{ \vec{x},\vec{u}\} \text{\ if and only if\ } 1\in (\dot{\vec{x}}-\vec{f},\vec{g},\dot{\vec{g}})
    \subseteq \C[\vec{x},\vec{u},\dot{\vec{x}},\dot{\vec{u}}].
\]
In other words, in order to obtain a (differential) representation of $1$ it suffices to derive once the algebraic reduced equations $\vec{g}$. $\square$
\end{proposition}

\section{The main case: ODE's with arbitrary algebraic constraint}\label{ODEarbitrary}

We will now consider semiexplicit differential systems with no restrictions on the dimension of the algebraic variety of constraints.

Let $\vec{x}=x_1, \ldots, x_n$ and $\vec{u}=u_1, \ldots, u_m$ be differential variables, and let   $\vec{f}=f_1, \ldots, f_n $ and $\vec{g}=g_1, \ldots, g_s $ be polynomials in $\C[\vec{x}, \vec{u}]$. We consider the differential first order semiexplicit system \begin{equation}\label{elsistema}
{\mathbf{\left\{\begin{array}{rcl}\dot{\vec{x}}-\vec{f}(\vec{x}, \vec{u})&=&0\\\vec{g}(\vec{x}, \vec{u})&=&0\end{array}\right.}}\end{equation}


Suppose that the differential system (\ref{elsistema}) has no solution (or equivalently, $1\in [\dot{\vec{x}}-\vec{f}, \vec{g}]\subseteq \C\{ \vec{x},\vec{u}\})$. Our goal is to find bounds for the order of a representation of $1$ as a combination of the polynomials $\dot{\vec{x}}-\vec{f}, \vec{g}$ and their derivatives. Without loss of generality we will suppose \emph{that the purely algebraic system $\vec{g}=0$ is consistent}, because if it is not, it suffices to write $1$ as a combination of the polynomials $\vec{g}$ and no derivatives are required.

The following theorem will be proved at the end of this section:

\begin{theorem} \label{elteorema} Let $\vec{x}=x_1, \ldots, x_n$ and $\vec{u}=u_1, \ldots, u_m$ be differential variables, and let  $\vec{f}=f_1, \ldots, f_n $ and $\vec{g}=g_1, \ldots, g_s $ be polynomials in $\C[\vec{x}, \vec{u}]$. Let $V\subset \C^{n+m}$ be the variety defined as the set of zeros of the ideal $(\vec{g})$,  $0\le r:=\dim(V)$, $\nu:=\max\{1, r\}$ and  $D$ be an upper bound for the degrees of $\vec{f}$, $\vec{g}$ and $V$. Then,  \[1\in [\dot{\vec{x}}-\vec{f}, \vec{g}]\quad \Longleftrightarrow \quad 1\in (\dot{\vec{x}}-\vec{f},\ldots ,  \vec{x}^{(L+1)}-\vec{f}^{(L)}, \vec{g}, \ldots , \vec{g}^{(L)}),\]
where $L\le ((n+m)D)^{2^{c\nu^2(n+m)}}$ for a universal constant $c>0$.
\end{theorem}

In what follows we will show how it is possible to obtain a system related to the original inconsistent input system (\ref{elsistema}) but whose algebraic variety of constraints has dimension $0$. To do this, we consider a sequence of auxiliary inconsistent differential systems such that their algebraic constraints define varieties with decreasing dimensions. Once this descending dimension process is done, we will be able to apply the results of  Section \ref{multivardim0}. Finally, we will estimate the order of derivatives of the equations which enable us to write $1$ as an element of the differential ideal by means of an ascending dimension process associated to the same sequence of auxiliary systems.

\subsection{The dimension descending process}

Let us begin by introducing some notation related to the differential part $\dot{\vec{x}}-\vec{f}=0$ of the system (\ref{elsistema}) that will be used throughout this section.

\begin{notation} \label{lastildes} If $\vec{h}=h_1, \ldots, h_\beta$ is a set of polynomials in $\C[\vec{x}, \vec{u}]$, we define \[\widetilde{h}_i:=\d{\sum_{j=1}^n\dfrac{\partial h_i}{\partial x_j}\ f_j+\sum_{k=1}^m\dfrac{\partial h_i}{\partial u_k}\ \dot{u}_k}\in \C[\vec{x}, \vec{u}, \dot{\vec{u}}]\]
for $i=1, \ldots , \beta$. In other words, \[\widetilde{\vec{h}}:=\dfrac{\partial \vec{h}}{\partial \vec{x}}\cdot \vec{f}+\dfrac{\partial \vec{h}}{\partial \vec{u}}\cdot \dot{\vec{u}}.\]
Note that the polynomials $\widetilde{\vec{h}}$ belong to the polynomial ideal $(\dot{\vec{x}}-\vec{f},\dot{\vec{h}})\cap \C[\vec{x},\vec{u},\dot{\vec{u}}]$.

If $I\subset \C[\vec{x},\vec{u}]$ is an ideal, we denote by $\widetilde{I}\subset \C[\vec{x},\vec{u},\dot{\vec{u}}]$ the ideal generated by $I$ and the polynomials $\widetilde{h}$ with $h\in I$.
 Note that if a set of polynomials ${\vec{h}}$ generates the ideal $I$, then the polynomials $\vec{h},\widetilde{\vec{h}}$ generate the ideal $\widetilde{I}$ in $\C[\vec{x},\vec{u},\dot{\vec{u}}]$ and that $\widetilde{I}\subset (\dot{\vec{x}}-\vec{f},\vec{h},\dot{\vec{h}})\cap \C[\vec{x},\vec{u},\dot{\vec{u}}]$.

\end{notation}

The key point to our dimension descending process is the following geometric Lemma.

\begin{lemma} \label{bajodim} Let $\pi:\C^{n+2m}\to\C^{n+m}$ be the projection $(\vec{x}, \vec{u}, \dot{\vec{u}})\mapsto (\vec{x}, \vec{u})$ and suppose that the ideal $(\vec{g})\subset \C[\vec{x}, \vec{u}]$ is radical.

If the system (\ref{elsistema}) has no solution, then no irreducible component of $V(\vec{g})$ is contained in the Zariski closure $\overline{\pi(V(\vec{g}, \widetilde{\vec{g}}))}$ and, in particular, since $\pi(V(\vec{g},\widetilde{\vec{g}}))\subseteq V(\vec{g})$, we have that $\dim \overline{\pi(V(\vec{g}, \widetilde{\vec{g}}))}< \dim V(\vec{g})$.
\end{lemma}

\begin{proof}  Throughout the proof, for a set of variables $\vec{z}$  and a set of polynomials $\vec{h}$  in $\C[\vec{z}]$, if $\overline{\vec{z}}\subset \vec{z}$ and $\overline{\vec{h}}\subset \vec{h}$, we will denote by $\dfrac{\partial \overline{\vec{h}}}{\partial \overline{\vec{z}}}$ the $\#(\overline{\vec{h}})\times \#(\overline{\vec{z}})$ Jacobian matrix of the polynomials  $\overline{\vec{h}}$ with respect to the variables $\overline{\vec{z}}$.

Suppose that there is an irreducible  component $C$ of $V(\vec{g})$ included in $\overline{\pi(V(\vec{g}, \widetilde{\vec{g}})})$. We will construct a solution for the system (\ref{elsistema}).

Since the Zariski algebraic closed set $\overline{{\pi(V(\vec{g}, \widetilde{\vec{g}})})}$ is contained in $V(\vec{g})$, there exists at least one irreducible component $Z$ of $V(\vec{g}, \widetilde{\vec{g}})$ such that $C=\overline{\pi(Z)}$. From Chevalley's Theorem  (see e.g. \cite[Ch.2, \S 6]{Matsumura80}) there exists a  nonempty Zariski open subset $\mathcal{U}$ of $C$ contained in the image $\pi(Z)\subseteq \pi(V(\vec{g}, \widetilde{\vec{g}}))$. Moreover, since the ideal $(\vec{g})$ is assumed to be radical, shrinking  the open set $\mathcal{U}$ if necessary, we may also suppose that all point $p\in \mathcal{U}$ is a regular point of $C$ and then, the Jacobian matrix $\dfrac{\partial \vec{g}}{\partial ({\vec{x},\vec{u}})}(p)$ has rank $n+m-\dim C$.

Similarly, we may suppose also that  for all $p\in \mathcal{U}$ the equality \[{\rm{rk}}\dfrac{\partial \vec{g}}{\partial \vec{u}}(p)=\max\left\{{\rm{rk}}\dfrac{\partial \vec{g}}{\partial \vec{u}}(x, u): (x, u)\in C\right\}\] holds, and that the first columns of $\dfrac{\partial \vec{g}}{\partial \vec{u}}(p)$ are a $\C$-basis of the column space of this matrix.

Set $l:=\textrm{ rk }\dfrac{\partial \vec{g}}{\partial \vec{u}}(p)$.
If $\widehat{\vec{u}}=u_1, \ldots, u_l$ and $\overline{\vec{u}}=u_{l+1}, \ldots, u_m$, then there exists a subset $\widehat{\vec{x}}\subseteq \vec{x}$ of cardinality $k:= n+m-\dim C-l$ such that  ${\rm{rk}}\dfrac{\partial \vec{g}}{\partial (\widehat{\vec{x}}, \widehat{\vec{u}})}(p)={\rm{rk}}\dfrac{\partial \vec{g}}{\partial ({\vec{x},\vec{u}})}(p)=k+l$. For simplicity assume that
$\widehat{\vec{x}}=x_1, \ldots, x_k$. We denote $\overline{\vec{x}}=x_{k+1}, \ldots , x_n$.\\

\noindent \emph{Claim\ }
For every  $p\in \mathcal{U}$ there exists a unique $\widehat{\eta}\in\C^l$ (depending  on $p$) such that $(p, \widehat{\eta}, 0)\in V(\vec{g}, \widetilde{\vec{g}})$.

\bigskip

\noindent \emph{\it Proof of the claim.}
Fix $p\in \mathcal{U}$. Since $\mathcal{U}\subseteq \pi(V(\vec{g}, \widetilde{\vec{g}}))$ there exists $(a, b)\in \C^{l}\times \C^{m-l}$ (not necessarily unique) such that $(p, a, b)\in V(\vec{g}, \widetilde{\vec{g}})$. Thus \[\dfrac{\partial \vec{g}}{\partial \vec{x}}(p)\cdot \vec{f}(p)+\dfrac{\partial \vec{g}}{\partial \widehat{\vec{u}}}(p)\cdot a+\dfrac{\partial \vec{g}}{\partial \overline{\vec{u}}}(p)\cdot b=0.\]
In particular, the linear system in the unknowns $(\vec{y},\vec{z})\in \C^{l}\times \C^{m-l}$: \[\dfrac{\partial \vec{g}}{\partial \widehat{\vec{u}}}(p)\cdot \vec{y}+\dfrac{\partial \vec{g}}{\partial \overline{\vec{u}}}(p)\cdot\vec{z}=- \dfrac{\partial \vec{g}}{\partial \vec{x}}(p)\cdot \vec{f}(p)\] has a solution, or  equivalently, the columns of $\dfrac{\partial \vec{g}}{\partial \vec{x}}(p)\cdot \vec{f}(p)$ belong to the linear subspace generated by the columns of the matrix $\dfrac{\partial \vec{g}}{\partial {\vec{u}}}(p)$. By our choice of the variables $\widehat{\vec{u}}$,  the columns of the matrix $\dfrac{\partial \vec{g}}{\partial \widehat{\vec{u}}}(p)$ are a basis of this subspace.
Then, the linear system $\dfrac{\partial \vec{g}}{\partial \widehat{\vec{u}}}(p)\cdot {\vec{y}}=-\dfrac{\partial \vec{g}}{\partial \vec{x}}(p)\cdot \vec{f}(p)$ has a unique solution $\widehat{\eta}\in \C^{l}$, which means that $(p, \widehat{\eta}, 0)\in V(\vec{g}, \widetilde{\vec{g}})$. This finishes the proof of the claim.

\bigskip

 Now we go back to the proof of the Lemma. We are looking for a solution of the system (\ref{elsistema}) when the irreducible component $C$ of the variety $V(\vec{g})$ lies in $\overline{\pi(V(\vec{g}, \widetilde{\vec{g}}))}$.

 Fix a point $(x_0, u_0)=(\widehat{x}_0, \overline{x}_0, \widehat{u}_0, \overline{u}_0)\in \mathcal{U}$. From the Implicit Function Theorem around $(x_0,u_0)$, shrinking the open set $\mathcal{U}$ in the strong topology if necessary, there exists a neighborhood $\mathcal{V}_0\subset \C^{(n-k)+(m-l)}$ of the point $(\overline{x}_0,\overline{u}_0)$ (also in the strong topology) and differentiable functions $\varphi_{1}:\mathcal{V}_0\to \C^{k}$ and $\varphi_{2}:\mathcal{V}_0\to \C^{l}$ such that for any $(x, u)=(\widehat{x}, \overline{x}, \widehat{u}, \overline{u})\in \mathcal{U}$, the equality  $(\widehat{x}, \widehat{u})=(\varphi_1(\overline{x}, \overline{u}),\varphi_2(\overline{x}, \overline{u}))$ holds and in particular we have
 \begin{equation} \label{impli}
(x,u)=(\varphi_1(\overline{x}, \overline{u}), \overline{x},\varphi_2(\overline{x}, \overline{u}),\overline{u}))\quad \textrm{for all}\ (x,u)\in \mathcal{U}.
 \end{equation}
For $i=k+1, \ldots,n$, we write  $\overline{{\psi}_i}(\overline{\vec{x}}, \overline{\vec{u}})=f_i(\varphi_1(\overline{\vec{x}}, \overline{\vec{u}}), \overline{\vec{x}}, \varphi_2(\overline{\vec{x}}, \overline{\vec{u}}), \overline{\vec{u}})$ and $\overline{\vec{\psi}}=\overline{\psi}_{k+1}, \ldots, \overline{\psi}_n$.

Let us consider the following ODE with initial condition in the $n-k$ unknowns $\overline{\vec{x}}$:
\begin{equation}\label{sistematilde} {\bf{\left\{\begin{array}{rcl}
\dot{\overline{\vec{x}}}\ \ &=&\overline{\vec{\psi}}(\overline{\vec{x}}, \overline{u}_0)\\\overline{\vec{x}}(0)&=&
\overline{x}_0\end{array}\right.}}\end{equation}
and let ${\gamma}(t)=({\gamma}_{k+1}(t), \ldots, {\gamma}_n(t))$ be a solution of the system  (\ref{sistematilde}) in a neighborhood of $0$.

From this solution $\gamma$ we define \[\Gamma(t)=(\varphi_1(\gamma(t),  \overline{u}_0), \gamma(t),  \varphi_2(\gamma(t), \overline{u}_0), \overline{u}_0)\]
in a neighborhood of $0$. From (\ref{sistematilde}) we have that $(\gamma(0),\overline{u}_0)=(\overline{x}_0,\overline{u}_0)$; therefore, the continuity of $\gamma$ ensures that for all $t$ in a neighborhood of $0$, $(\gamma(t),\overline{u}_0)$ belongs to the open set $\mathcal{V}_0$ and then $\Gamma$ is well defined.

It suffices to prove that $\Gamma(t)$ is a solution of the original system (\ref{elsistema}), which leads to a contradiction since that system has no solution.

First of all, by (\ref{impli}) we deduce that $\Gamma(t)\in \mathcal{U}\subseteq C$ for all $t$ small enough and so, $\vec{g}(\Gamma(t))=0$. In other words, $\Gamma(t)$ satisfies the algebraic constraint of the system (\ref{elsistema}).
In order to show that $\Gamma(t)$ also satisfies the differential part of (\ref{elsistema}) we observe that
its coordinates $k+1,\ldots,n$ are simply $\gamma(t)$, which satisfy the differential relations:
\[
\dfrac{d}{dt}(\gamma_i(t))=\overline{\psi}_i(\gamma(t),\overline{u}_0)
\]
for $i=k+1,\ldots,n$. Since $\overline{\psi}_i(\gamma(t),\overline{u}_0)=f_i(\Gamma(t))$, we conclude that $\Gamma$ satisfies the last $n-k$ differential equations of (\ref{elsistema}).
It remains to show that it also satisfies the first $k$ differential equations of (\ref{elsistema}).

Taking the derivative with respect to the single variable $t$ in the identity $\vec{g}(\Gamma(t))=0$ we obtain:
\[\dfrac{\partial \vec{g}}{\partial \widehat{\vec{x}}}(\Gamma(t))\cdot \dfrac{d}{dt}(\varphi_1(\gamma(t),\overline{u}_0))+\dfrac{\partial \vec{g}}{\partial \overline{\vec{x}}}(\Gamma(t))\cdot \dfrac{d}{dt}({\gamma(t)})+\dfrac{\partial \vec{g}}{\partial \widehat{\vec{u}}}(\Gamma(t))\cdot \dfrac{d}{dt}(\varphi_2(\gamma(t), \overline{u}_0))=0,\]
and then
\begin{equation}\label{gamma1}-\dfrac{\partial \vec{g}}{\partial \overline{\vec{x}}}(\Gamma(t))\cdot \dfrac{d}{dt}({\gamma(t)})=\dfrac{\partial \vec{g}}{\partial (\widehat{\vec{x}},\widehat{\vec{u}})}(\Gamma(t))
\cdot \dfrac{d}{dt}(\varphi(\gamma(t),\overline{u}_0)),\end{equation}
where $\varphi=(\varphi_1,\varphi_2)$.

On the other hand, since for all $t$ in a neighborhood of $0$ we have $\Gamma(t)\in \mathcal{U}\subseteq \pi(V(\vec{g}, \widetilde{\vec{g}}))$, the previous {\emph {Claim}} implies that there exist a unique ${\widehat{\eta}}(t)\in \C^l$ such that $(\Gamma(t), {\widehat{\eta}}(t), 0)\in V(\vec{g}, \widetilde{\vec{g}})$ and then, if we write $\widehat{\vec{f}}=f_1, \ldots , f_k$ and $\overline{\vec{f}}=f_{k+1}, \ldots, f_n$,
\[\dfrac{\partial \vec{g}}{\partial \widehat{\vec{x}}}(\Gamma(t))\cdot \widehat{\vec{f}}(\Gamma(t))+\dfrac{\partial \vec{g}}{\partial \overline{\vec{x}}}(\Gamma(t))\cdot \overline{\vec{f}}(\Gamma(t))+\dfrac{\partial \vec{g}}{\partial \widehat{\vec{u}}}(\Gamma(t))\cdot {\widehat{\eta}}(t)=0.\]
Hence,
\begin{equation}\label{gamma2}-\dfrac{\partial \vec{g}}{\partial \overline{\vec{x}}}(\Gamma(t))\cdot \overline{\vec{f}}(\Gamma(t))=\dfrac{\partial \vec{g}}{\partial (\widehat{\vec{x}},\widehat{\vec{u}})}(\Gamma(t))\cdot (\widehat{\vec{f}}(\Gamma(t)), {\widehat{\eta}}(t)).\end{equation}

Since $\dfrac{d}{dt}({\gamma(t)})=\overline{\psi}(\gamma(t),\overline{u}_0)= \overline{\vec{f}}(\Gamma(t))$ and the matrix $\dfrac{\partial \vec{g}}{\partial (\widehat{\vec{x}},\widehat{\vec{u}})}(\Gamma(t))$ has a $(k+l)\times(k+l)$ invertible minor, comparing (\ref{gamma1}) and (\ref{gamma2}), we infer that
\[\dfrac{d}{dt}(\varphi(\gamma(t),\overline{u}_0))=
(\widehat{\vec{f}}(\Gamma(t)), {\widehat{\eta}}(t)).\]
In particular $\dfrac{d}{dt}(\varphi_1(\gamma(t),\overline{u}_0))=
\widehat{\vec{f}}(\Gamma(t))$. Thus $\Gamma$ verifies also the first $k$ differential equations in (\ref{elsistema}) and the lemma is proved. \end{proof}\\

Before we begin the descending process, let us remark that the new differential system induced by the construction underlying Lemma \ref{bajodim} trivially inherits the inconsistency from the input system (\ref{elsistema}):

\begin{remark} \label{sistemamenor} Let notations be as in Lemma
\ref{bajodim} and  $\vec{g}_1$ be a set of generators of the
(radical) ideal $I(\overline{\pi(V(\vec{g},
\widetilde{\vec{g}}))})$. If the system (\ref{elsistema}) has no
solution, neither does the differential system
$$\left\{\begin{array}{rcl}\dot{\vec{x}}-\vec{f}(\vec{x},
\vec{u})&=&0\\\vec{g}_1(\vec{x}, \vec{u})&=&0\end{array}\right. .$$
This is a consequence of the inclusion of algebraic ideals
$(\vec{g})\subset \sqrt{(\vec{g})}\subset (\vec{g}_1)$, which
implies the inclusion of differential ideals $[\vec{g}]\subset
[\vec{g}_1]$; hence, $1\in [\dot{\vec{x}}-\vec{f}, \vec{g}_1]$.
\end{remark}

We will now begin to present the descending dimension process induced by Lemma \ref{bajodim}.

\begin{definition} \label{lasgis} From the
system (\ref{elsistema}), we define recursively an increasing chain
of radical ideals $I_0\subset I_1\subset\cdots $ in the polynomial
ring $\C[\vec{x},\vec{u}]$ as follows:
\begin{itemize}
\item $I_0=\sqrt {(\vec{g})}$.
\item  Assuming that $I_i$ is defined, consider the ideal
$\widetilde{I}_{i}\subseteq \C[\vec{x}, \vec{u}, \dot{\vec{u}}]$
introduced in Notation \ref{lastildes} and suppose that $1\notin
\widetilde{I}_i$. We define $I_{i+1} =\sqrt{\widetilde{I}_i \cap
\C[\vec{x}, \vec{u}]}$.
\end{itemize}
\end{definition}

Let us observe some basic facts about this definition. First, note
that $I_0=I(V(\vec{g}))$ and, for $i\ge 1$, if $\pi:\C^{n+2m}\to
\C^{n+m}$ is the projection $(\vec{x}, \vec{u},
\dot{\vec{u}})\mapsto (\vec{x}, \vec{u})$, then
$I_{i+1}=I(\overline{\pi(V(\widetilde{I}_i))})$ (we have that
$V(\widetilde{I}_i)\subset \C^{n+2m}$ is nonempty since we assume
$1\notin \widetilde{I}_i$).
Secondly, from Lemma \ref{bajodim}, 
it follows  that the chain of ideals defined is strictly increasing
since the inequality $\dim(V(I_{i+1}))<\dim(V(I_i))$ holds. We can
estimate the length of this chain: if we define
\[\rho:=\min\{i\in \mathbb{Z}_{\ge 0} : \dim (V(I_i))\le 0\}\]
we have that $0\le \rho \le r$ (recall that we assume that the ideal $I_0$ is a
proper ideal with $\dim(V(I_0))=r\ge 0$). Notice also that if
$\rho>0$, then $\dim (V(I_\rho))=0$ or $-1$ and $\dim (V(I_{\rho-1}))>0$.\\

 Any system of generators $\vec{g}_\rho$ of the last ideal $I_\rho$ allows us to exhibit a new ODE system, related to the original one, with no solutions and such that $1$ can be easily written  as a combination of the generators of the associated differential ideal. More precisely,

\begin{proposition} \label{con0dim} Fix $i=0,\ldots,\rho$ and let $\vec{g}_{i}\subset \C[\vec{x},\vec{u}]$ be any system of generators of the  radical ideal $I_i$. Then the DAE system \[\left\{\begin{array}{rcl}\dot{\vec{x}}-\vec{f}(\vec{x}, \vec{u})&=&0\\ \vec{g}_{i}(\vec{x},\vec{u})&=&0\end{array}\right. \]
has no solution.
In the particular case $i=\rho$, we also  have that $1$ belongs to the polynomial ideal $(\dot{\vec{x}}-\vec{f}, \vec{g}_{\rho}, \dot{\vec{g}}_{\rho})\subset \C[\vec{x}, \dot{\vec{x}}, \vec{u}, \dot{\vec{u}}]$.
\end{proposition}

\begin{proof}
The first assertion follows from the iterated application of Remark \ref{sistemamenor}. If $i=\rho$ and $\dim(I_\rho)=-1$ (i.e. if $I_\rho=\C[\vec{x},\vec{u}]$) we have $1\in (\vec{g}_\rho)\subset (\dot{\vec{x}}-\vec{f}, \vec{g}_{\rho}, \dot{\vec{g}}_{\rho})$. Otherwise, if $\dim(I_\rho)=0$, since the ideal $I_{\rho}=(\vec{g}_\rho)$ is radical, the Proposition follows from the $0$-dimensional case considered in Proposition \ref{dimension 0}.
\end{proof}\\

In the following Lemma  we will estimate bounds for the degree and the number of polynomials in suitable families  $\vec{g}_i$ which generate the ideals $I_i$, for each $i=1,\ldots,\rho$. It will be a consequence of Proposition \ref{prop:radbound}.

\begin{lemma}\label{gradogis} Consider the DAE system (\ref{elsistema}), and let $r:=\dim(V(\vec{g}))$,  $\nu:=\max\{1, r\}$ and $D:=\max\{\deg(\vec{f}), \deg(\vec{g}),\deg(V(\vec{g}))\}$. There exists a universal constant $c>0$ such that for each $0\le i\le \rho$, the ideal $I_i$ can be generated by a family of polynomials $\vec{g}_i$ whose number and degrees are bounded by $((n+m)D)^{2^{c(i+1)\nu (n+m)}}$. Moreover, if $r>0$, this is also an upper bound for the degrees of the polynomials $\widetilde{\vec{g}}_i$.
\end{lemma}

\begin{proof}
First, notice that, if $r=0$, then $\rho=0$. Since $\vec{g}_0$ is a set of generators of $\sqrt{(\vec{g})}$, the bound is a direct consequence of Proposition \ref{prop:radbound}. Let us now suppose that $r>0$.

From Proposition \ref{prop:radbound}, since $V(\vec{g})\subset
\C^{n+m}$ is an algebraic variety of dimension $r_0:=r$ and degree
at most $D_0:= D$, the radical ideal $I_0=I(V(\vec{g}))$ can be generated
by a set $\vec{g}_0$ of $\delta_0:=((n+m)D)^{2^{c_0 r_0(n+m)}}$
polynomials of degrees at most $\delta_0$, where $c_0$ is an
adequate positive constant. By modifying the constant $c_0$ if
necessary, we may suppose that $\delta_0$ is also an upper bound for
the degrees of the polynomials  $\widetilde{\vec{g}}_0$ introduced
in Notation \ref{lastildes}.

Following \cite[Lemma 2]{Heintz83}, the degree of the
variety $\overline{\pi(V(\widetilde{I}_0))}=
\overline{\pi(V(\vec{g}_0, \widetilde{\vec{g}}_0))}$ is at most $
\deg(V(\vec{g}_0, \widetilde{\vec{g}}_0))$ and then, from
Proposition \ref{prop:bezbound} and the previous bounds we infer that
\[\deg(\overline{\pi(V(\widetilde{I}_0))})\le \deg(V(\vec{g}_0, \widetilde{\vec{g}}_0))\le D_0 \delta_0^{r_0} = (n+m)^{r_0 2^{c_0 r_0 (n+m)}} D^{1+r_0 2^{c_0 r_0 (n+m)}}=:D_1.\]
Applying again the estimate stated in Proposition
\ref{prop:radbound}, we have that the ideal
$I_1=I(\overline{\pi(V(\widetilde{I}_0)})$ can be generated by a set
$\vec{g}_1$ consisting of at most \[\delta_1:=((n+m)D_1)^{2^{c_0
r_1 (n+m)}} = ((n+m)D)^{2^{c_0r_1(n+m)}(1+r_0 2^{c_0
r_0(n+m)})}\] polynomials of degrees bounded by $\delta_1$, where
$r_1:= \dim(V(I_1))<r_0$. By the choice of $c_0$, $\delta_1$ is also an upper bound for the degrees of the polynomials $\widetilde{\vec{g}}_1$.

Proceeding in the same way, it can be proved inductively that, for
every $0\le i \le \rho-1$,
\[\deg(\overline{\pi(V(\widetilde{I}_i))})\le\deg (V(\vec{g}_i,\widetilde{\vec{g}}_i))\le
D_i \delta_i^{r_i} = \frac{1}{n+m} ((n+m) D)^{\prod_{j=0}^{i}
(1+r_j 2^{c_0 r_j (n+m)})}=:D_{i+1}\]
and, therefore, the radical ideal $I_{i+1}=I(\overline{
\pi(V(\widetilde{I}_i))})$ can be generated by a
system of polynomials $\vec{g}_{i+1}$ whose number and degrees are
bounded by
\[\delta_{i+1}:=((n+m) D_{i+1})^{2^{c_0r_{i+1}(n+m)}} = ((n+m) D)^{2^{c_0 r_{i+1}(n+m)}\prod_{j=0}^{i} (1+r_j 2^{c_0 r_j (n+m)})},\]
which is also an upper bound for the degrees of the polynomials $\widetilde{\vec{g}}_i$.

The result follows by choosing a universal constant $c$ such that $1+r 2^{c_0r(n+m)} \le 2^{cr(n+m)}$ for every $r> 0$.
\end{proof}

\bigskip

We introduce the following invariants associated to the chain of ideals $I_0\subset I_1\subset \cdots \subset I_\rho$ and the generators $\vec{g}_i$, $i=0, \ldots, \rho$, considered in Lemma \ref{gradogis}, that we will use in the next section.

\begin{definition}\label{lasep}
With the previous notations, for each $i=0,\ldots,\rho$, we define $\varepsilon_i$ as follows:
\[\varepsilon_0:=\min\{\varepsilon\in \N : I_0^{\varepsilon}\subseteq (\vec{g})\}, \]
\[\varepsilon_i:=\min\{\varepsilon\in \N : I_i^{\varepsilon}\subseteq (\dot{\vec{x}}-\vec{f}, \vec{g}_{i-1}, \dot{\vec{g}}_{i-1})\}\quad \textrm{for\ }\ i>0.\]
\end{definition}

Observe that by definition, $\varepsilon_i>0$ for all $i$. Moreover, they are well defined (i.e. finite) since $I_i$ is the radical of the ideal $\widetilde{I}_{i-1}\cap \C[\vec{x},\vec{u}]$ and $\widetilde{I}_{i-1}=(\vec{g}_{i-1},\widetilde{\vec{g}}_{i-1})\subseteq (\dot{\vec{x}}-\vec{f},\vec{g}_{i-1},\dot{\vec{g}}_{i-1})$ (see the final remark in Notation \ref{lastildes}). In fact, we can obtain upper bounds for these integer numbers in terms of the parameters of the input system (\ref{elsistema}):

\begin{proposition}\label{subengs} As in Lemma \ref{gradogis}, let $D:=\max\{\deg(\vec{f}), \deg(\vec{g}),\deg(V(\vec{g}))\}$. We have the inequalities:
\[\varepsilon_0\le D^{n+m}\qquad \textrm{and} \qquad \varepsilon_i\le ((n+m)D)^{2^{cir(n+m)}} \ \textrm{for}\ i=1,\ldots,\rho,\]
where $c>0$ is a universal constant.
\end{proposition}

\begin{proof}
If $i=0$ the bound follows from Proposition \ref{prop:exprad} applied to  $I=(\vec{g})$ and $I_0=\sqrt{I}$ in the polynomial ring $\C[\vec{x},\vec{u}]$: we have that $I_0^{{\deg(\vec{g})^{n+m}}}\!\!\!\!\!\subset (\vec{g})$ and then, $\varepsilon_0\le D^{n+m}$.

Now, fix an index $i=1,\ldots,\rho$. From Definition \ref{lasgis},
it follows that $I_i$ is contained in
$\sqrt{(\vec{g}_{i-1},\widetilde{\vec{g}}_{i-1})}\subset
\C[\vec{x},\vec{u},\dot{\vec{u}}]$. Thus, applying Proposition
\ref{prop:exprad} to the ideal
$I=(\vec{g}_{i-1},\widetilde{\vec{g}}_{i-1})$, with the estimations
of Lemma  \ref{gradogis}, we conclude that $I_i^\varepsilon \subset
(\vec{g}_{i-1},\widetilde{\vec{g}}_{i-1})$ for
$\varepsilon:=(((n+m)D)^{2^{cir(n+m)}})^{n+2m}=((n+m)D)^{(n+2m)2^{cir(n+m)}}$,
since in this case $r>0$ and then $\nu=r$.

The proposition follows from the fact that
$(\vec{g}_{i-1},\widetilde{\vec{g}}_{i-1})\subset
(\dot{\vec{x}}-\vec{f}, \vec{g}_{i-1},\dot{\vec{g}}_{i-1})$ changing
the constant $c$ by another one $c'$ such that the inequality
$(n+2m)2^{cir(n+m)}\le 2^{c'ir(n+m)}$ holds for any $n,m$.
\end{proof}

\subsection{Going Back to the Original System} \label{back}

We follow the notations and keep the hypotheses of the previous
section.

The aim of this section is to prove the main Theorem
\ref{elteorema}, roughly speaking, an upper bound for the number of
derivations needed to obtain a representation of $1$ as an element
of the differential ideal $[\dot{\vec{x}}-\vec{f}, \vec{g}]$
introduced in formula (\ref{elsistema}).

Let $\vec{g}_i$, $i=0,\ldots,\rho$, be the polynomials in
$\C[\vec{x},\vec{u}]$ introduced in Lemma \ref{gradogis}. From
Proposition \ref{con0dim}, the differential Nullstellensatz asserts
that $1\in [\dot{\vec{x}}-\vec{f}, \vec{g}_i]$ for
$i=0,\ldots,\rho$. Thus for each $i$ there exists a non negative
integer $k$ (depending on $i$) such that
$1\in((\dot{\vec{x}}-\vec{f})^{[k]}, \vec{g}_i^{[k]})\subseteq
\C[\vec{x}^{[k+1]},\vec{u}^{[k]}]$.

For $i=0,\ldots,\rho$, we define:
\begin{equation}\label{defki}
 k_i:=\min\{k\in \N_0 : 1\in((\dot{\vec{x}}-\vec{f})^{[k]},
\vec{g}_i^{[k]})\}. \end{equation}
Observe that, since
$(\vec{g}_i)=I_i\subset I_{i+1}=(\vec{g}_{i+1})$,  the sequence
$k_i$ is decreasing and that Proposition \ref{con0dim} ensures that
$k_\rho\le 1$.

The following key lemma allows us to bound recursively each
$k_{i-1}$ in terms of $k_{i}$ with the help of the sequence
$\varepsilon _i$ introduced in Definition \ref{lasep}:

\begin{lemma}\label{cotaks} Suppose that the finite sequence $k_i$ introduced in
(\ref{defki}) is not identically zero and let $\mu:=\max \{ 0\le
i\le \rho : k_i\ne 0\}$. Then:
\begin{enumerate}
\item the inequality $k_{i-1}\le 1+\varepsilon_i k_i$ holds for every $1\le i\le \mu$.
\item $k_{\mu}=1$ and $k_i=0$ for every $i>\mu$.
\item $\d{k_0\le (\mu+1)\prod_{i=1}^{\mu} \varepsilon_i}$.
\end{enumerate}
\end{lemma}

\begin{proof}
Obviously, $k_i=0$ for all $i> \mu$ from the definition of $\mu$.

Consider now $1\le i\le \mu+1$. From Definition \ref{lasep}, any
$g\in I_i=(\vec{g}_i)$ satisfies
\begin{equation} \label{powerg}
g^{\varepsilon_i}\in (\dot{\vec{x}}-\vec{f}, \vec{g}_{i-1},
\dot{\vec{g}}_{i-1}).\end{equation}

For any $j\ge 1$, if we differentiate $j\varepsilon_i$ times the
polynomial $g^{\varepsilon_i}$, by means of  Leibniz's Formula,  we
have
\[(g^{\varepsilon_i})^{(j\varepsilon_i)}=\sum_{r_1+r_2+\ldots r_{\varepsilon_i}=j\varepsilon_i}{j\varepsilon_i\choose r_1\ldots r_{\varepsilon_i}}\, g^{(r_1)}\dots g^{(r_{\varepsilon_i})},\] where $r_1, \ldots , r_{\varepsilon_i}$ are nonnegative integers and ${j\varepsilon_i\choose r_1\ldots r_{\varepsilon_i}}:=\dfrac{(j\varepsilon_i)!}{r_1!\ldots r_{\varepsilon_i}!}$.
In the formula above, the term with $r_1=\ldots
=r_{\varepsilon_i}=j$ equals
 $\dfrac{(j\varepsilon_i)!}{(j!)^{\varepsilon_i}}(g^{(j)})^{\varepsilon_i}$ and,
since the sum of all the
 $r_l$, with $l=1, \ldots , \varepsilon_i$, must be
 $j\varepsilon_i$,
  in all the other terms there is  at least one $g^{(r_l)}$ with $r_l<j$. Thus, for every $j=0, \ldots, k_i$, there exist polynomials $p_{j,0}, \ldots, p_{j, j-1}$,  such that the equality \[(g^{\varepsilon_i})^{(j\varepsilon_i)}=\dfrac{(j\varepsilon_i)!}{(j!)^{\varepsilon_i}}
  \ ({g}^{(j)})^{\epsilon_i}+\sum_{l=0}^{j-1}p_{j,l}\,  g^{(l)}\]
 holds;
moreover, from (\ref{powerg}), by differentiating $j\varepsilon_i$
times we deduce that
\[\dfrac{(j\varepsilon_i)!}{(j!)^{\varepsilon_i}}
  \ ({g}^{(j)})^{\varepsilon_i}+\sum_{l=0}^{j-1}p_{j,l}\,  g^{(l)}\in ((\dot{\vec{x}}-\vec{f})^{[j\varepsilon_i]}, \vec{g}_{i-1}^{[1+j\varepsilon_i]}).\]

From these relations we deduce recursively that
a common zero of the polynomials
$(\dot{\vec{x}}-\vec{f})^{[k_i\varepsilon_i]}$ and $
\vec{g}_{i-1}^{[1+k_i\varepsilon_i]}$ is also a zero of the
polynomials $g, \dot{g}, \ldots, g^{(k_i)}$ for every $g\in I_i=
(\vec{g}_i)$; in particular, it is a common zero of
$(\dot{\vec{x}}-\vec{f})^{[k_i]}, \vec{g}_{i}^{[k_i]}$ (recall that,
by definition, $\varepsilon_i$ is at least $1$ and then
$k_i\varepsilon_i\ge k_i$ for all $i$). This contradicts the
definition of $k_i$ (see (\ref{defki})).

Therefore, $1\in ((\dot{\vec{x}}-\vec{f})^{[k_i\varepsilon_i]},
\vec{g}_{i-1}^{[1+k_i\varepsilon_i]})$ and then, $k_{i-1}\le
1+k_i\varepsilon_i$, which proves the first part
of the lemma.

In particular, if $\mu<\rho$, taking $i=\mu+1$ in the previous
inequality, we have $0<k_{\mu}\le 1+k_{\mu+1}\varepsilon_{\mu+1}=1$,
hence $k_{\mu}=1$. In the case $\mu=\rho$, Proposition \ref{con0dim}
implies that $k_\rho\le 1$ and then, since $k_\mu$ is assumed
nonzero, we have $k_\mu=k_\rho=1$. This proves the second assertion.

Finally, the last statement is an easy consequence of the previous
ones: it follows from the fact that $k_\mu=1$,
applying recursively the inequality $k_{i-1}\le 1+ k_i\varepsilon_i
$ for $i=\mu, \mu-1,\dots, 1$.
\end{proof}\\

We are now ready to prove Theorem \ref{elteorema} stated at the beginning of this subsection as a corollary of the previous lemma and its proof:\\

\noindent {\bf{Proof of Theorem \ref{elteorema}.}} We must estimate
an upper bound for the minimum integer $L$ such that
\[1\in ((\dot{\vec{x}}-\vec{f})^{[L]},
\vec{g}^{[L]}).\]
By Definition \ref{lasep} the inclusion
$(\vec{g}_0)^{\varepsilon_0}\subseteq (\vec{g})$ holds. We can
repeat the same argument as in the proof of Lemma \ref{cotaks} and
prove that, for any $g\in (\vec{g}_0)=I_0=\sqrt{(\vec{g})}$ and for
$j=0, \ldots ,k_0$, there exist polynomials $p_{j,0},\ldots, p_{j,
j-1}$ such that
\[(g^{\varepsilon_0})^{(j\varepsilon_0)}=\dfrac{(j\varepsilon_0)!}{(j!)^{\varepsilon_0}}
  \ ({g}^{(j)})^{\varepsilon_0}+\sum_{l=0}^{j-1}p_{j,l}\,
g^{(l)}\in (\vec{g}^{[j\varepsilon_0]}).\] As before, this implies
that a common zero of the polynomials
$(\dot{\vec{x}}-\vec{f})^{[ k_0\varepsilon_0]}$ and $\vec{g}^{[
k_0\varepsilon_0]}$ is also a zero of  $g, \ldots, g^{(k_0)}$ for an
arbitrary element $g\in (\vec{g}_0)$; in particular, taking into
account that $\varepsilon_0\ge 1$, it follows that it is a common
zero of $(\dot{\vec{x}}-\vec{f})^{[k_0]}, \vec{g}_{0}^{[k_0]}$,
contradicting the definition of $k_0$ (see (\ref{defki})). We
conclude then that
\[1 \in ((\dot{\vec{x}}-\vec{f})^{[ k_0\varepsilon_0]}, \vec{g}^{[ k_0\varepsilon_0]}).\]
Thus, the inequality $L\le  k_0\varepsilon_0 $ holds.

If $r=0$, from Proposition \ref{dimension 0}, we have that $k_0=1$
and then, from Proposition \ref{subengs},  $L\le\varepsilon_0\le
D^{n+m}$. On the other hand, if $r>0$, from Proposition
\ref{subengs} and Lemma \ref{cotaks}, taking into account that
$\mu\le r$, we have that there is a constant $c$ such that

\[\begin{array}{cl}L\le  k_0\varepsilon_0 &\le  D^{n+m} (\mu+1) \prod_{i=1}^{\mu}
((n+m)D)^{2^{c i r(n+m)}}=\\[2mm]
& = (\mu+1) ((n+m)D)^{\sum_{i=0}^{\mu}2^{cir(n+m)}}\le  \\[2mm]
&\le (\mu+1)((n+m)D)^{2^{1+c\mu r(n+m)}}\le \\[2mm]
&\le (r+1)((n+m)D)^{2^{1+c r^2(n+m)}}.\end{array}\] The desired
bound follows taking a new constant $c'$, in such a way that the
inequality  $(r+1)((n+m)D)^{2^{1+c r^2(n+m)}}\le
((n+m)D)^{2^{c'r^2(n+m)}}$ holds.
Finally, since $\nu=\max\{1, r\}$, it is easy to see that $D^{n+m}$
and $((n+m)D)^{2^{c'r^2(n+m)}}$ are both bounded by
 $((n+m)D)^{2^{c{''} \nu^2(n+m)}}$ for a suitable universal constant $c''>0$.

The converse  is obvious. $\square$\\

As we have observed in the previous proof, if we assume $r=0$
(i.e.~the algebraic constraint $\vec{g}=0$ of the DAE system
(\ref{elsistema}) defines a $0$-dimensional algebraic variety in
$\C^{n+m}$), the constant $L$ can be bounded directly by
$\varepsilon_0$. This estimation is optimal for this kind of DAE
systems as it is shown by the following example extracted from
\cite[Example 5]{GKOS}:

\begin{example}
In the DAE system (\ref{elsistema}) suppose $n=1$, $\vec{x}=x_1$, $\vec{u}=u_1,\ldots,u_m$, $\vec{f}=1$,  and $\vec{g}(\vec{x},\vec{u})=u_m-x_1^2, u_{m-1}-u_m^2, \ldots, u_1-u_2^2,u_1^2$.

\noindent As it is shown in \cite[Example
5]{GKOS}, this system has no solutions and $1\in
((\dot{\vec{x}}-\vec{f})^{[L]},\vec{g}^{[L]})$ if and only if $L\ge
2^{m+1}$. On the other hand, the inequality $\varepsilon_0\le
2^{m+1}$ holds from Proposition \ref{prop:exprad}. Hence
$\varepsilon_0=2^{m+1}$ and the upper bound is reached.
\end{example}

\section{The general case} \label{generalcase}

The well-known method of reducing the order of a system of
differential equations will enable us to apply the results of the
previous Section to the general case.

Let $\vec{x}=x_1, \ldots, x_n$ and  $\vec{f}=f_1, \ldots, f_s $ be
differential  polynomials in $\C[\vec{x}, \ldots, \vec{x}^{(e)}]$,
with $e\ge 1$. We  consider the differential system
\begin{equation}\label{elsistemaordene}
{\mathbf{\left\{\begin{array}{rcl}f_1(\vec{x}, \ldots,
\vec{x}^{(e)})&=&0\\&\vdots &\\ f_s(\vec{x}, \ldots,
\vec{x}^{(e)})&=&0\end{array}\right.}}\end{equation}

Theorem \ref{elteorema} yields the following:

\begin{theorem}Let $\vec{x}=x_1, \ldots, x_n$ and  $\vec{f}=f_1, \ldots, f_s $ be
differential polynomials in $\C[\vec{x}, \ldots,
\vec{x}^{(e)}]$. Let $V\subset \mathbb{C}^{n(e+1)}$ be the algebraic
variety defined by $\{\vec{f}=0\}$, and let 
$\nu:=\max\{1, \dim(V)\}$ and $D:=\max\{\deg({\vec{g}}), \deg(V)\}$.
Then
$$1\in [\vec{f}] \ \Longleftrightarrow \ 1\in (\vec{f}, \ldots,
\vec{f}^{(L)})$$ where $L\le (n(e+1)D)^{2^{c\nu^2 n(e+1)}}$ for
 a universal constant $c>0$.
\end{theorem}
\begin{proof}
As usual, the introduction of the new of variables
$z_{i,j}:=x_i^{(j)}$, for $i=1, \ldots, n$ and $j=0, \ldots, e$, and
$\vec{z}_j=z_{1,j}, \ldots, z_{n, j}$, allows us to transform the
implicit system (\ref{elsistemaordene}) into the following first
order system with a set of polynomial constraints:
\begin{equation}\label{elsistemaorden1}
{\mathbf{\left\{\begin{array}{rcl}\dot{\vec{z}}_{0}&=&\vec{z}_1
\\\vdots \\\dot{\vec{z}}_{e-1}&=&\vec{z}_e \\
\overline{\vec{f}}&=&0\end{array}\right.}}\end{equation} where
$\overline{\vec{f}}=f_1(\vec{z}_0, \ldots, \vec{z}_{e}),\ldots ,
f_s(\vec{z}_0, \ldots, \vec{z}_{e})$.
It is clear that $$1\in[\vec{f}] \iff 1\in
[\dot{\vec{z}}_{0}-\vec{z}_1 , \ldots ,
\dot{\vec{z}}_{e-1}-\vec{z}_e, \overline{\vec{f}}].$$

The differential part of the system (\ref{elsistemaorden1}) is given
by an ODE consisting of $ne$ equations in $n(e+1)$ variables and the
constraints are given by polynomials in $n(e+1)$ variables, that is,
$\overline{\vec{f}}\in \C[\vec{z}_0, \ldots, \vec{z}_{e}]$.
Then, Theorem
\ref{elteorema} applied to this system
implies that $$1\in [\dot{\vec{z}}_{0}-\vec{z}_1 , \ldots ,
\dot{\vec{z}}_{e-1}-\vec{z}_e,
\overline{\vec{f}}]\Longleftrightarrow 1\in
((\dot{\vec{z}}_{0}-\vec{z}_1)^{[L]} , \ldots ,
(\dot{\vec{z}}_{e-1}-\vec{z}_e)^{[L]},  \overline{\vec{f}}^{[L]})$$
where $L\le (n(e+1)D)^{2^{c\nu^2(n(e+1))}}$, with $c>0$  a universal
constant. Going back to the original variables, replacing
$z_{i,j}^{(k)}$ by $x_i^{(j+k)}$ for $i=1, \ldots, n$, $j=0, \ldots
e$ and $k=0, \ldots, L$,  we get that
\[1\in [\vec{f}]\Longleftrightarrow 1\in (\vec{f}, \ldots , \vec{f}^{(L)})\]
and the Theorem follows.
\end{proof}

\bigskip

Applying the trivial bound $\dim(V)\le n(e+1)$ and Bezout's inequality, which implies that $\deg(V)\le d^{n(e+1)}$, we
deduce the following purely syntactic upper bound for the order in
the differential Nullstellensatz:

\begin{corollary}\label{sintactico}
Let $\vec{f}\subset \C\{\vec{x}\}$ be a finite set of differential polynomials in the variables $\vec{x}=x_1,\ldots,x_n$, whose degrees and orders are bounded by $d$ and $e$ respectively.
Then,
\[1\in [\vec{f}]\Longleftrightarrow 1\in (\vec{f}, \ldots , \vec{f}^{(L)})\]
where $L\le (n(e+1)d)^{2^{c(n(e+1))^3}}$ for a universal constant $c>0$. $\square$
\end{corollary}

Now, once the order is bounded, as a straightforward consequence of the classical effective Nullstellensatz (see for instance \cite[Theorem 1.1]{Jelonek05}), we can estimate the degrees of the polynomials involved in the representation of $1$ as an element of the ideal $[\vec{f}]$.

\begin{corollary}\label{losgrados} Let $\vec{f}=\{f_1,\ldots,f_s\}\subset \C\{\vec{x}\}$ be a finite set of differential polynomials in the variables $\vec{x}=x_1,\ldots,x_n$, whose degrees and orders are bounded by $d$ and $e$ respectively. Let $\epsilon:=\max\{2,e\}$. Then $1\in[\,\vec{f}\,]$ if, and only if, there exist polynomials $p_{ij}\in \C[\vec{x}^{[\epsilon +L]}]$
such that $\d{1=\sum_{i=1}^s\sum_{j=0}^L p_{ij}\, f_i^{(j)}}$, where
\[L\le (n\epsilon d)^{2^{c(n\epsilon)^3}}\qquad \textrm{and}\qquad   \deg(p_{ij}\,
f_i^{(j)})\le d^{(n\epsilon d)^{2^{c(n\epsilon)^3}}},\]
for a universal constant $c>0$.
\end{corollary}

\begin{proof} This corollary is an immediate consequence of the result in
\cite[Theorem 1.1]{Jelonek05}) applied to the polynomials
$\vec{f}^{[L]}$ of Corollary \ref{sintactico}, once we notice that
all the polynomials $\vec{f}^{[L]}$ have degree bounded by $d$,
since differentiation does not increase the degree, and that
$N:=n(e+L+1)$ is the number of variables used.
Thus $\deg(p_{ij}f_i^{(j)})\le 2 d^{N}$, and we only need to change the constant $c$ for a new constant $c'>0$ such that  $(n(e+1)d)^{2^{c(n(e+1))^3}}\le (n\epsilon d)^{2^{c'(n\epsilon)^3}}$ and,  for $d\ge 2$,
$2d^{n(e+(n(e+1) d)^{2^{c(n(e+1))^3}}+1)}\le d^{(n\epsilon d)^{2^{c'(n\epsilon)^3}}}$. For $d=1$, we can take $p_{ij}\in \C$ and so, the degree upper bound also holds. \end{proof}

\bigskip

We remark that Corollary \ref{losgrados} allows us to construct an
algorithm which decides if an ordinary DAE system $\vec{f}=0$ over
$\C$ has a solution or not. It suffices to consider the coefficients
of the polynomials $p_{ij}$ as indeterminates  (finitely many, since
orders and degrees are bounded \textit{a priori}) and  obtain them
by solving
a non homogeneous linear system over $\C$.
It is easy to see that the complexity of this procedure becomes triply exponential
in the parameters $n$ and $e$. Another algorithm of the same hierarchy of
complexity (i.e.~triply exponential) can be deduced as a particular case of the
quantifier elimination method  of ordinary differential equations proposed by D. Grigoriev in \cite{Grigoriev87}.

\bigskip

In the usual way, we can deduce an effective strong differential Nullstellensatz from Corollary \ref{sintactico} and the well-known Rabinowitsch trick:


\begin{corollary} \label{nssstrong}
Let $\vec{f}\subset
\C\{\vec{x}\}$ be a finite set of differential polynomials in the
variables $\vec{x}=x_1,\ldots,x_n$. Suppose that
$f\in\C\{\vec{x}\}$ is a differential polynomial such that every
solution of the differential system $\vec{f}=0$ is a solution of the
differential equation ${f}=0$. Let $d:=\max\{\deg(\vec{f}),
\deg(f)\}$ and $\epsilon:=\max\{2, {\rm{ord}}(\vec{f}), {\rm{ord}}(f)\}.$ Then
$f^{M}\in (\vec{f}^{[L]})\subset [\,\vec{f}\,]$ where
$M=d^{n(\epsilon+L+1)}$ and $L\le (n\epsilon d)^{2^{c(n\epsilon)^3}}$ for a universal constant $c>0$.

\end{corollary}

\begin{proof}
We start with the Rabinowitsch trick: since every solution of the
system $\vec{f}=0$ is a solution of  $f=0$, if we introduce a new
differential variable $y$, the differential system $
\vec{f}=0\,,\, 1-yf=0$ has no solution. Hence, $1$ belongs to the
differential ideal $[\vec{f}, 1-yf]\subseteq \C\{\vec{x}, y\}$.

Therefore, Corollary \ref{sintactico} implies that $1$ belongs to the
polynomial ideal $(\vec{f}^{[L]}, (1-yf)^{[L]})$, with $L\le
((n+1)(\epsilon +1)(d+1))^{2^{c((n+1)(\epsilon +1))^3}}\le (n\epsilon d)^{2^{c'(n\epsilon)^3}}$, for suitable universal constants $c, c'>0$. Taking any
representation of $1$ as a linear combination of the generators
$\vec{f}^{[L]}, (1-yf)^{[L]}$ with polynomial coefficients and
replacing each variable $y^{(i)}$ by the corresponding
$(f^{-1})^{(i)}$ for $0\le i\le L$, we deduce that a suitable power
of $f$ belongs to the polynomial ideal $(\vec{f}^{[L]})$ or,
equivalently, $f\in \sqrt{(\vec{f}^{[L]})}$ in the polynomial ring
$\C[\vec{x}^{[\epsilon+L]}]$.

Now, applying \cite[Theorem 1.3]{Jelonek05} stated in our
Proposition \ref{prop:exprad}, we conclude that $f^{d^N}\in
(\vec{f}^{[L]})$ for $N:=n(\epsilon +L+1)$, the number of variables of the
ground polynomial ring $\C[\vec{x}^{[\epsilon +L]}]$, and the corollary is
proved.
\end{proof}

\bigskip

With the same notation and assumptions as in Corollary \ref{nssstrong},
we can obtain upper bounds for the degrees of the polynomials involved in
a representation of a power of $f$ as an element of the differential ideal $[\vec{f}]$. Applying \cite[Corollary 1.7]{Kollar88}, we have that, if $\vec{f}=f_1,\dots, f_s$ are differential polynomials
of \emph{degrees at least $3$}, there are polynomials $p_{ij}\in \C[\vec{x}^{[\epsilon +L]}]$ such that $\d f^{M} =\sum_{i=1}^s\sum_{j=0}^L p_{ij}\, f_i^{(j)}$ with $\deg(p_{ij}\, f_i^{(j)}) \le (1+d) M$, where $M= d^{n(\epsilon +L+1)}$.
For differential polynomials with \emph{arbitrary degrees}, following the proof of the first part of Corollary \ref{nssstrong} and taking into account the bounds in Corollary \ref{losgrados}, we get a representation $\d f^{\widetilde M} = \sum_{i=1}^s\sum_{j=0}^L \widetilde p_{ij}\, f_i^{(j)}$, where $\widetilde M = 2 (L+1) d^{n(\epsilon +L+1)+1}\le d^{(n\epsilon d)^{2^{c(n\epsilon)^3}}} $, with $\deg (\widetilde p_{ij})\le d^{(n\epsilon d)^{2^{c(n\epsilon)^3}}} $ for a suitable universal constant $c>0$.

\section{The case of an arbitrary field of constants} \label{general}

In the previous sections the assumption of taking the complex numbers as the ground field is only essential in the proof of Lemma \ref{bajodim} because the Implicit Function Theorem is applied. Even if the statement of this lemma makes sense in any field, we are not able to prove it in the more general case of any field of constants.
However, it is not difficult to prove that if $K$ is an arbitrary field of characteristic $0$ (with the trivial derivation) the following analogue of Theorem \ref{elteorema} remains true for $K$. More precisely:

\begin{theorem} \label{elteoremaK}
Let $K$ be an arbitrary field of characteristic $0$ with the trivial derivation, $\vec{x}=x_1, \ldots, x_n$ and $\vec{u}=u_1, \ldots, u_m$ differential variables over $K$,  $\vec{f}=f_1, \ldots, f_n $ and $\vec{g}=g_1, \ldots, g_s $ polynomials in $K[\vec{x}, \vec{u}]$. If $d>0$ is an upper bound for the degrees of $\vec{f}$ and $\vec{g}$ we have:  \[1\in [\dot{\vec{x}}-\vec{f}, \vec{g}]\subseteq K\{\vec{x},\vec{u}\}\quad \Longleftrightarrow \quad 1\in (\dot{\vec{x}}-\vec{f},\ldots ,  \vec{x}^{(L+1)}-\vec{f}^{(L)}, \vec{g}, \ldots , \vec{g}^{(L)}),\]
where $L\le ((n+m)d)^{2^{c(n+m)^3}}$ for a suitable universal constant $c>0$.
\end{theorem}

\begin{proof}
We prove only the non trivial implication. Suppose that
\begin{equation*} \label{en K}
1\in [\dot{\vec{x}}-\vec{f}, \vec{g}]\subseteq K\{\vec{x},\vec{u}\}
\end{equation*}
holds.
Let $\vec{k}\subseteq K$ be the subfield of $K$ which is generated over $\Q$ by the coefficients of the finitely many polynomials involved in a decomposition of $1$ as an element of the ideal $[\dot{\vec{x}}-\vec{f}, \vec{g}]$.
Obviously, we have $1\in [\dot{\vec{x}}-\vec{f}, \vec{g}]$ in the differential ring $\vec{k}\{\vec{x},\vec{u}\}$.

Since $\vec{k}$ is a finitely generated extension of $\Q$ and $\C$ is algebraically closed having infinite transcendence degree over $\Q$, there exists a field $\Q$-embedding  $\sigma: \vec{k}\to \C$. This morphism can be extended to an embedding of the ring $\vec{k}\{\vec{x}, \vec{u}\}$ into $\C\{\vec{x}, \vec{u}\}$, simply by applying $\sigma$ to the coefficients of the polynomials. Moreover, by considering $\vec{k}$ and $\C$ as fields of constants, $\sigma$ defines a  monomorphism of differential rings. In particular, the condition $1\in [\dot{\vec{x}}-\vec{f}, \vec{g}]\subseteq \vec{k}\{\vec{x},\vec{u}\}$ implies that $1\in [\dot{\vec{x}}-\sigma(\vec{f}), \sigma(\vec{g})]\subseteq \sigma(\vec{k})\{\vec{x},\vec{u}\}\subseteq \C\{\vec{x},\vec{u}\}$.

Now, applying Theorem \ref{elteorema} to the polynomials $\sigma(\vec{f}), \sigma(\vec{g})$ we conclude that the relation
\begin{equation} \label{en K2}
1\in  (\dot{\vec{x}}-\sigma(\vec{f}),\ldots ,  \vec{x}^{(L+1)}-\sigma(\vec{f})^{(L)}, \sigma(\vec{g}), \ldots , \sigma(\vec{g})^{(L)})
\end{equation}
holds in an algebraic polynomial ring $\C[\vec{x}^{[L+1]},\vec{u}^{[L]}]$ with $L\le ((n+m)D)^{2^{c \nu^2(n+m)}}$, where $D$ is the degree of the variety $V\subset \C^{n+m}$ defined by $\sigma(\vec{g})$ and $\nu:=\max\{1, \dim(V)\}$. Using B\'ezout's inequality we may replace $D$ by $d^{n+m}$ and then, changing the constant $c$ if necessary, we may suppose $L\le ((n+m)d)^{2^{c(n+m)^3}}$. Moreover, we can also replace $\C$ by the subfield $\sigma(\vec{k})$, because the coefficients of the polynomial combination underlying (\ref{en K2}) may be chosen as solutions of a linear system with coefficients in that field. The theorem follows by taking $\sigma^{-1}$ in that polynomial relation.
\end{proof}

\bigskip

Since all the statements in Section \ref{generalcase} follow by straightforward manipulations of the equations plus the bound stated in Theorem \ref{elteorema}, we can replace $\C$ by any arbitrary base field of constants $K$ of characteristic $0$ and the results of that section remain true.

\bigskip

\noindent \textbf{Acknowledgments.} The authors thank Michael Wibmer (RWTH Aachen University) for pointing out a way of extending the results from the field of complex numbers to an arbitrary field of constants of characteristic zero.


\begin{thebibliography}{00}

\bibitem{BS91} Berenstein, C.; Struppa, D. Recent Improvements in the Complexity of the Effective Nullstellensatz. Linear Algebra Appl. 157 (1991), 203--215.


\bibitem{Cohn} Cohn, R. On the analogue for differential equations of the Hilbert-Netto theorem  Bull. Amer. Math. Soc. Volume 47, Number 4 (1941), 268--270.

\bibitem{Dube90} Dub\'e, T.
The structure of polynomial ideals and Gr\"obner bases.
SIAM J. Comput. 19, no. 4 (1990), 750--775.

\bibitem{Giusti84}
Giusti, M.
 Some effectivity problems in polynomial ideal theory. EUROSAM 84 (Cambridge, 1984),
Lecture Notes in Comput. Sci., 174, Springer, Berlin (1984), 159--171.

\bibitem{Grigoriev87} Grigoriev, D. Complexity of quantifier elimination in the theory of ordinary differential equations. EUROCAL '87 (Leipzig, 1987), Lecture Notes in Comput. Sci., 378, Springer, Berlin (1989),  11--25.

\bibitem{GKOS} Golubitsky, O.;  Kondratieva, M.;   Ovchinnikov, A.;  Szanto, A. A bound for orders in differential Nullstellensatz. Journal of Algebra 322, no. 11 (2009), 3852--3877

\bibitem{Heintz83} Heintz, J. Definability and fast quantifier elimination in algebraically closed fields.
Theoret. Comput. Sci. 24, no. 3 (1983), 239--277.

\bibitem{HS82}  Heintz, J.; Schnorr, C.
 Testing polynomials which are easy to compute. Logic and algorithmic (Zurich, 1980),
Monograph. Enseign. Math., 30, Univ. Gen\`eve (1982), 237--254.

\bibitem{GH} Hermann, G. Die Frage der endlich vielen Schritte in der Theorie der Polynomideale. Math.
Ann. 95 (1926), 736--788.

\bibitem{Jelonek05} Jelonek, Z. On the effective Nullstellensatz. Invent. Math. 162, no. 1 (2005), 1--17.

\bibitem{Johnson02} Johnson, W. The Curious History of Fa\`a di Bruno's Formula. The Amer. Math. Monthly, Vol. 109, No. 3 (2002), 217--234

\bibitem{Kolchin} Kolchin, E. Differential Algebra and Algebraic Groups. Academic Press, NewYork, (1973).


\bibitem{Kollar88} Koll\'ar, J. Sharp effective Nullstellensatz. J. Amer. Math. Soc. 1 (1988), no. 4, 963--975.

\bibitem{KL91a} Krick, T.; Logar, A.
 Membership problem, representation problem and the computation of the radical for one-dimensional ideals. Effective methods in algebraic geometry (Castiglioncello, 1990), Progr. Math., 94, Birkhäuser Boston, Boston, MA (1991),  203--216.

\bibitem{KL91b} Krick, T.; Logar, A.
 An algorithm for the computation of the radical of an ideal in the ring of polynomials. Applied algebra, algebraic algorithms and error-correcting codes (New Orleans, LA, 1991),
Lecture Notes in Comput. Sci., 539, Springer, Berlin (1991), 195--205.

\bibitem{Kronecker1882} Kronecker, L. Grundz\"uge einer arithmetischen Theorie der algebraischen Gr\"ossen. J. Reine Angew. Math. 92 (1882), 1--123.

\bibitem{Laplagne06} Laplagne, S.
 An algorithm for the computation of the radical of an ideal. ISSAC 2006, ACM, New York (2006),  191--195.

\bibitem{Matsumura80} Matsumura, H. Commutative Algebra. Second Edition. The Benjamin/Cummings Publ. Company (1980).

\bibitem{MM84} M\"oller, M.; Mora, F.
 Upper and lower bounds for the degree of Groebner bases. EUROSAM 84 (Cambridge, 1984),
Lecture Notes in Comput. Sci., 174, Springer, Berlin (1984), 172--83.

\bibitem{Raud34} Raudenbush, H. Ideal theory and algebraic differential equations. Trans. Amer. Math. Soc. vol. 36 (1934), 361--368.

\bibitem{Ritt32}
Ritt, J.  Differential equations from the algebraic standpoint. Amer. Math. Soc. Colloq. Publ., Vol XIV, New York (1932).

\bibitem{Ritt50} Ritt, J. Differential Algebra. Amer. Math. Soc. Coll. Publ. Vol. 33, New York (1950).


\bibitem{Seid52} Seidenberg, A. Some basic theorems in differential algebra (characteristic $p$
arbitrary). Trans. Amer. Math. Soc. 73 (1952), 174--190 .

\bibitem{Seid56} Seidenberg, A. An elimination theory for differential algebra. Univ. Calif. Publ. Math. (New
Series) 3 (1956), 31--65 .

\end{thebibliography}
\end{document}